\documentclass{amsart}
\usepackage{chngcntr}
\usepackage{apptools}
\usepackage{color}
\usepackage{amsthm,amssymb,verbatim}
\usepackage{graphicx}
\usepackage{enumerate}
\usepackage{chngcntr}
\usepackage{apptools}
\usepackage{color}
\usepackage{amsthm,amssymb,verbatim}
\usepackage{graphicx}
\usepackage{enumerate}

 \newcommand{\Ee}{\mathcal{E}}

  \newcommand{\Ss}{\mathbf{S}}
 
 \newcommand{\RR}{\mathbf{R}}  
 \newcommand{\BB}{\mathbf{B}}  

    \newcommand{\dist}{\operatorname{dist}}

 \newcommand{\Tan}{\operatorname{Tan}}
 
 \newcommand{\Kk}{\mathcal{K}}

\newcommand{\Hh}{\mathcal{H}}
\usepackage{amsthm}

\input epsf
\def\begfig {
\begin{figure}
\small }
\def\endfig {
\normalsize
\end{figure}
}

    \newtheorem{theorem}    {Theorem}   

    \newtheorem*{theorem*}{Theorem}
    \theoremstyle{definition}
    \newtheorem{definition}  [theorem] {Definition}
     
    \theoremstyle{definition}
    \newtheorem{remark}   [theorem]       {Remark}

\title[Sharp Entropy Bounds]{Sharp Entropy Bounds 
for Self-Shrinkers \\
in Mean Curvature Flow}
\author{Or Hershkovits }
\address{Department of Mathematics\\ Stanford University\\ Stanford, CA 94305}
\email{orher@stanford.edu}
\author{Brian White}
\address{Department of Mathematics\\ Stanford University\\ Stanford, CA 94305}
\email{bcwhite@stanford.edu}
\subjclass[2010]{Primary 53C44; Secondary 49Q20.}
\date{March 1, 2018.  Revised October 13, 2018.}
\usepackage{hyperref}
\usepackage{enumerate}
\usepackage[alphabetic, msc-links, backrefs]{amsrefs}
\begin{document}

\maketitle

\begin{abstract}
Let $M\subset \RR^{m+1}$ be a smooth, closed, codimension-one  self-shrinker (for mean curvature flow)
with nontrivial $k^{\rm th}$ homology.  We show that the entropy
of $M$ is greater than or equal to the entropy of a round $k$-sphere,
and that if equality holds, then $M$ is a round $k$-sphere in $\RR^{k+1}$.
\end{abstract}

\section{introduction}
A properly embedded hypersurface $M\subset \RR^{m+1}$ is called a self-shrinker if $M_t:=\sqrt{-t}M$ for $t\in (-\infty,0)$ is an evolution by mean curvature, i.e., if
\[
(\partial_t x)^{\perp}=\vec{H}(x)
\] 
holds for every $t\in (-\infty,0)$ and $x\in M_t$. Equivalently, $M$ is a self-shrinker if it satisfies
\begin{equation}\label{self_shrink_eq}
\vec{H}+x^{\perp}/2=0.
\end{equation}
  
The study of self-shrinkers is central in the analysis of 
singularity formation of the mean curvature flow (MCF). 
Indeed, every limit of rescalings of a MCF around a fixed point in spacetime is modeled on a 
possibly singular self-shrinker~\cite{Hui_mon,White_stra,Ilm_sur}. 
It is straightforward to check
that a hyperplane through the origin is a self-shrinker,
as is $\Ss^k(\sqrt{2k})$, the $k$-sphere of radius $\sqrt{2k}$ in $\RR^{k+1}$.
Crossing with a plane through the origin leaves  equation~\eqref{self_shrink_eq} unchanged, 
so the cylinder $\Ss^k(\sqrt{2k})\times \RR^{m-k}$ in $\RR^{m+1}$ is  also a self-shrinker. 
We regard spheres as a special cases of cylinders: $\Ss^k=\Ss^k\times \RR^0$.
Although many other self-shrinkers have been constructed \cite{Ang_tor,Kap_shri,Ket_shri},
 Huisken~\cite{Ilm_ques} conjectured that for MCF from generic initial hypersurfaces, 
 all singularities are cylindrical. 
 When the initial hypersurface is mean-convex, 
 all singularities are indeed cylindrical~\cite{Hui_mon, Hui_Sin, White_nature}.
 
In a recent fundamental paper \cite{CM_generic}, Colding and Minicozzi 
 made an important step towards establishing {Huisken's} genericity conjecture. In that paper,
they defined the {\bf Gaussian area} of a hypersurface $M$ in $\RR^{m+1}$ to be
\begin{equation}\label{F_def}
F[M]=\frac{1}{(4\pi)^{m/2}}\int_M e^{-|x|^2/4}d\mathcal{H}^m,
\end{equation}
and they defined its {\bf entropy} to be the {supremum of}
the Gaussian area of all transates and rescalings of $M$:
\begin{equation}\label{ent_def}
\Ee[M]=\sup_{x_0\in \RR^{m+1}, \lambda>0} F[\lambda(x-x_0)].
\end{equation}
Clearly, {Gaussian area} is invariant under rotations, 
and {entropy} is invariant under all rigid motions and rescalings.  
The normalization constant $\frac{1}{(4\pi)^{m/2}}$ in the definition of $F$ 
is {chosen} so that linear hyperplanes have Gaussian area $1$.
It follows that $F[M]=F[M\times \RR]$ for every $M$, and 
thus that $\Ee[M]=\Ee[M\times\RR]$.  

Entropy is related to MCF through Huisken's monotonicity formula \cite{Hui_mon}, which implies that entropy is non-increasing under the flow.
{
Moreover, given a MCF with initial surface $N$, 
if $M$ is a self-shrinker that arises (as discussed above)
by blowing-up around a singular point of the flow, then} 
\begin{equation}\label{lower_bound_by_sing}
F[M]=\Ee[M]\leq \Ee[N].
\end{equation}  
(The first equality holds for every self-shrinker, as was shown in \cite{CM_generic}). The main result of \cite{CM_generic} states that every self-shrinker $M$ other than the spheres and cylinders can be 
 perturbed to a hypersurface with lower entropy. 
Thus by~ \eqref{lower_bound_by_sing}, if we flow from the perturbed hypersurface, then
$M$ cannot appear as a singularity model. 
 
In \cite{Stone}, Stone calculated the $F$-areas of shrinking spheres (and thus also of shrinking
cylinders).  By ~\eqref{lower_bound_by_sing}, 
those $F$-areas are the entropies of round spheres.
According to those calculations,
\begin{equation}\label{monot_ent}
2> \Ee[\Ss^1]> \Ee[\Ss^2]> \dots,
    \quad\text{and}\quad 
    \lim_{n\to\infty} \Ee[\Ss^{n}] = \sqrt{2}.
\end{equation} 

In this paper, we prove:
\begin{theorem}\label{intro-main-theorem}
Suppose that $M\subset \RR^{m+1}$ is a codimension-one, smooth, closed self-shrinker 
with nontrivial $k^{\rm th}$ homology.  Then the entropy
of $M$ is greater than or equal to the entropy of a round $k$-sphere.
If equality holds, then $M$ is a round $k$-sphere in $\RR^{k+1}$.
\end{theorem}
The special case $k=m$ is the main result of~\cite{CIMW}.
The special cases $(k,m)=(1,2)$ and $(k,m)=(2,3)$ of Theorem~\ref{intro-main-theorem}
follow from recent work of Jacob Bernstein and Lu Wang.
Indeed, \cite{BW_top_prop} proves that
any smooth closed hypersurface in $\RR^3$ with entropy less than $\Ee(\Ss^1)$
is isotopic to $\Ss^2$, and \cite{BW_top_clo} proves
that any smooth closed hypersurface in $\RR^4$ with
entropy less  than $\Ee(\Ss^2)$ is diffeomorphic to $\Ss^3$.

\subsection*{Acknowledgments}
Or Hershkovits was partially supported by an AMS-Simons Travel Grant.
Brian White was partially supported by NSF grants~DMS-1404282 and~DMS-1711293.

\section{Proof of the Sharp Entropy Bounds}
\begin{theorem}\label{main-homotopy}
Suppose that $M\subset \RR^{m+1}$ is a codimension-one, smooth, closed self-shrinker,
and that
 one of the components of $\RR^{m+1}\setminus M$ has nontrivial $k^{\rm th}$ homotopy.
Then the entropy of $M$ is greater
than or equal to the entropy of a round $k$-sphere.
If equality holds, then $M$ is a round $k$-sphere in $\RR^{k+1}$.
\end{theorem}
Before proving Theorem \ref{main-homotopy}, {we} show {that} it implies our main theorem.
\begin{proof}[Proof of Theorem \ref{intro-main-theorem}]
By Mayer-Vietoris, one of the components of the complement has nontrivial $k^{\rm th}$ homology.
By the Hurewicz~Theorem, that component has nontrivial $j^{\rm th}$ homotopy for some $j\le k$.
By Theorem~\ref{main-homotopy},
\[
    \Ee(M) \ge \Ee(\Ss^j),
\]
with equality if and only if $M$ is a round $j$-sphere in $\RR^{j+1}$.
The result follows immediately since $\Ee(\Ss^j)> \Ee(\Ss^k)$ for $j< k$ by~\eqref{monot_ent}.
\end{proof}
\begin{proof}[Proof of Theorem~\ref{main-homotopy}]
Consider the vectorfield
\begin{align*}
&X: \RR^{m+1}\to\RR^{m+1}, \\
&X(x) = x/2.
\end{align*}
We say that a region $K$ of $\RR^{m+1}$ with smooth boundary is strictly {\bf $X$-mean-convex}
if $H + X^\perp$ is nonzero and points into $K$ at each point 
of $\partial K$, where $H$ is the mean curvature
vector of $\partial K$.
Let $\Omega$ be a component of $\RR^{m+1}\setminus M$ such that $\RR^{m+1}\setminus \Omega$
has nontrivial $k^{\rm th}$ homotopy.  
We may suppose that $M$ is not a round sphere (otherwise the result is trivially true.)
By \cite[Lemma 1.2]{CIMW}, we can deform $M$ by pushing it slightly into $\Omega$
to get a surface $M'\subset \Omega$ such that
\begin{equation}\label{E-E'-entropy}
    \Ee(M') < \Ee(M)
\end{equation}
and such that 
 $K'$ is strictly $X$-mean-convex, where $K'$
is the closure of the component of 
$\RR^{m+1}\setminus M'$ that is
contained in $\Omega$.
Now as $M'$ is a smooth hypersurface, we can let it evolve for short time by $X$-mean curvature flow, i.e.,
with normal velocity $H+X^\perp$. 
Since $H+X^\perp$ points into $K'$, the surface immediately moves into the interior of $K'$. 
In fact, as explained in Section~\ref{prelim} (see Definition~\ref{weak_set_def} 
and Theorem~\ref{X-flow}), we can { extend the flow to all $t\ge 0$ 
(in particular, past singularities) by  letting
\[
   t\in [0,\infty)\mapsto M'(t)
\]
be the weak $X$-mean-curvature flow with $M'(0)=M'$.}  
For the particular vectorfield $X(x)=x/2$ we are using, $X$-mean-curvature
flow is also called {\bf renormalized mean curvature flow}: 
it differs from the ordinary mean curvature flow by a spacetime change-of-coordinates.  
To be precise, given our weak $X$-mean curvature flow $M'(\cdot)$,
the flow
\begin{equation}\label{tilde-M-flow}
  \tilde M: t\in [-1,0) \mapsto \tilde M(t) =  \sqrt{-t} M'(-\log|t|).
\end{equation}
is a weak mean curvature flow in $\RR^{m+1}$ with initial surface $\tilde M(-1)=M'$.
This is because  \eqref{tilde-M-flow} transforms smooth $X$-mean curvature flows to smooth
mean curvature flows, and hence weak $X$-mean curvature flows to weak mean curvature flows,
since the weak flows are defined by avoidance with smooth flows.
Note also that Huisken's monotonicity formula implies a modified monotonicity for $\tilde M(\cdot)$,
and that existence of tangent flows to $M'(\cdot)$ implies existence of tangent flows to $M'(\cdot)$.
Indeed the tangent flows to $M'(\cdot)$ at a specified spacetime point are the same as the tangent
flows to $\tilde M(\cdot)$ at the corresponding spacetime point.

Just as in the mean-convex setting, we can think of $t\mapsto M'(t)$ as a flow of measures
(see~Theorem~\ref{X-flow}).
Since entropy decreases under the flow $t\mapsto \tilde M(t)$, we see that
it also decreases under the renormalized 
flow $t\mapsto M'(t)$.
Consequently, if $\Theta$ is the Gauss density at a spacetime point of the flow $t\mapsto M'(t)$, then
\begin{equation}\label{gauss-density-fact}
   \Theta \le \Ee(M'(0)) = \Ee(M').
\end{equation}
Now let 
\[
   t\in [0,\infty)\mapsto K'(t).
\]
be the weak $X$-mean-curvature flow with $K'(0)=K'$ (see Definition~\ref{weak_set_def}). 
By Theorem~\ref{X-flow},
\[
\partial K'(t)=M'(t).
\]
Since $\RR^{m+1}\setminus K'$ has nontrivial $k^{\rm th}$ homotopy,
there is a continuous map $F: \partial \BB^{k+1}\to \RR^{m+1}\setminus K'$
such that $F$ is homotopically nontrivial in $\RR^{m+1}\setminus K'$.
Extend $F$ to a continuous map
\[
   F: \overline{\BB^{k+1}}\to \RR^{m+1}.
\]

By Theorem~\ref{clearing-out-theorem}, $F(\BB^{k+1})\cap K'(T)=\emptyset$ for $T$ sufficiently large. 
Since $F|\partial \BB^{k+1}$ is homotopically nontrivial in $\RR^{m+1}\setminus K'(0)$
and homotopically trivial in $\RR^{m+1}\setminus K'(T)$, the flow must be singular 
at one or more intermediate times.
In fact, the $X$-mean-convexity implies more (see Theorem~\ref{morse}):
there is a $t\in (0,T)$ and
an $x\in M(t)$ such that the tangent flow at $(x,t)$ is a shrinking $\Ss^j\times\RR^{m-j}$ for some $j \le k$.
Consequently, the Gauss density $\Theta$ at that point is
\[ 
  \Theta = \Ee(\Ss^j\times\RR^{m-j}) = \Ee(\Ss^j).
\]
Hence by~\eqref{monot_ent}, \eqref{E-E'-entropy}, and~\eqref{gauss-density-fact}, 
\[
  \Ee(\Ss^k) \le \Ee(\Ss^j) = \Theta \le \Ee(M') < \Ee(M).
\]
\end{proof}

\section{Motion by Mean Curvature Plus an Ambient Vectorfield}\label{prelim}

In this section we define weak $X$-mean curvature flow of closed sets,
and we state precisely the properties of the flow that were used in the proof of 
Theorem~\ref{intro-main-theorem}.

The following definition is an adaptation of the ones in \cite{White_top_weak ,Ilm_mani}.

\begin{definition}\label{weak_set_def}
Suppose that $K$ is a closed subset of $\RR^{m+1}$ and that $X$ is a smooth
vectorfield on $\RR^{m+1}$.  Let $\Kk$ be the largest closed subset of $\RR^{m+1}\times [0,\infty)$
such that
\begin{enumerate}[\upshape \quad(1)]
\item\label{initial-condition-item} $K(0)=K$, and
\item\label{avoidance-item} If $t\in [a,b]\subset [0,\infty) \mapsto \Delta(t)$ 
is a $X$-mean-curvature flow of smooth, compact hypersurfaces with $\Delta(a)$ disjoint from $K(a)$,
then $\Delta(t)$ is disjoint from $K(t)$ for all $t\in [a,b]$,
\end{enumerate}
where
\[
   K(t):=\{x\in\RR^{m+1}: (x,t)\in \Kk\}.
\]
We say that
\[
  t\in [0,\infty)\mapsto K(t)
\]
is the {\bf weak $X$-mean-curvature-flow} (or simply the {\bf weak $X$-flow}) starting from $K$.
\end{definition}

The largest set $\Kk$ exists because the closure of the union of all sets $\Kk$
having properties (1) and (2) also has those properties.

For a self-contained treatment of weak $X$-flows, see~\cite{HW-avoidance}.
(In that paper any closed subset of $\mathcal{K}$ satisfying~\eqref{initial-condition-item} and~\eqref{avoidance-item}
is called a {\bf weak $X$-flow} starting from $K$, and the largest one is called {\bf the biggest $X$-flow} starting
from $K$.  In this paper, the only weak $X$-flow starting from $K$ that we need is the biggest one,
and we write ``the weak $X$-flow" rather than ``the biggest $X$-flow''.)

The following theorem lists the main properties
of weak $X$-flow of $X$-mean-convex regions.

\begin{theorem}\label{X-flow}
Suppose that $K$ is a closed region in $\RR^{m+1}$ with smooth, compact boundary.
Suppose that $X$ is a smooth vector field on $\RR^{m+1}$ such that
\[
  \sup_x \frac{|X(x)|}{|x|+1} < \infty.  \tag{*}
\]
and such that at each point of $\partial K$, the vector $\overrightarrow{H}+X^\perp$ is nonzero
and points into $K$.
Let $t\mapsto K(t)$ and $t\mapsto M(t)$ be 
the weak-$X$-flows starting from $K$ and from $\partial K$.
Then 
\begin{enumerate}[\upshape \quad(1)]
\item\label{subset-item} $K(t_2)\subset \mathrm{Int}(K(t_1))$ 
     whenever $0\leq t_1<t_2<\infty$. 
\item\label{boundary-item} $M(t)=\partial K(t)$ for each $t<\infty$.
\item\label{compact-item} $M(t)$ is compact for each $t<\infty$.
\item\label{unit-regular-item} $t\in [0,\infty)\mapsto \Hh^m\llcorner M(t)$ defines a unit-regular integral $X$-brakke flow.
\item\label{parabolic-item} The flow $t\mapsto M(t)$ is smooth away from 
a closed set of parabolic Hausdorff dimension $(m-1)$ in spacetime.
\item\label{convex-type-item} The singular points of the flow $t\mapsto M(t)$ are of convex type.
\end{enumerate}
\end{theorem}

A spacetime singular point $(x,t)$ is said to be of {\bf convex type}
provided the following holds: if $x_i\in M(t_i)$ are regular points with $(x_i,t_i)\to (x,t)$, then the mean curvature $h_i$
of $M(t_i)$ at $x_i$ tends to infinity, and $h_i(M(t_i)-x_i)$ converges smoothly
(after passing to a subsequence)
to a convex hypersurface $M'$ of $\Tan(N,x_i)$.
(Here we regard $N$ as isometrically embedded in some Euclidean space.)

For definition of ``$X$-brakke flow", see~\cite{HW-avoidance}*{\S12} or~\cite{HW_X_mean_convex}.

The hypothesis~\thetag{*} guarantees that compactness is preserved, i.e., that $\cup_{t\in [0,T]}M(t)$ is compact for finite $T$.
See~\cite{HW-avoidance}*{Theorem~23}.
More generally, in smooth Riemannian manifolds and without the hypothesis~\thetag{*},
the conclusions of the theorem hold as long as $\cup_{t\in [0,T]}M(t)$ is compact.

In the case of Euclidean space with no vectorfield (i.e., $X=0$), 
Theorem~\ref{X-flow} was proved in~\cite{White_size,White_nature,White_sub}.
That work was extended to compact $K$ in Riemannian manifolds (still with $X\equiv 0$)
by Haslhofer and Hershkovits~\cite{HH}.
The proof of Theorem~\ref{X-flow} is a modification of the proofs in those papers.
See~\cite{HW-avoidance} for proofs of Assertions~\eqref{subset-item}, 
   \eqref{boundary-item}, and~\eqref{compact-item}, 
   and~\cite{HW_X_mean_convex} for proofs of Assertions \eqref{unit-regular-item},
    \eqref{parabolic-item}, and~\eqref{convex-type-item}.

\begin{remark}
Although the proof of Theorem~\ref{intro-main-theorem} only used the vector field $X=x/2$, in order to prove Theorem \ref{X-flow} for this particular vector field when $K$ is unbounded (which is key to identifying interior topology in Theorem~\ref{intro-main-theorem}), one is forced to consider more general vector fields. 
Thus, from the point of view of this current paper,
it is (indirectly) essential that  
the analysis in \cite{HW-avoidance, HW_X_mean_convex} holds
for arbitrary vector fields $X$ satisfying ~\thetag{*}, and not just for $X=x/2$.
\end{remark}

\begin{theorem}\label{morse}
Suppose in Theorem~\ref{X-flow} that
\[
   F: \partial{\BB^{k+1}} \to \RR^{m+1}\setminus K
\]
is homotopically nontrivial in $\RR^{m+1}\setminus K$ and homotopically trivial in
 $\RR^{m+1}\setminus K(T)$.   Then there is a $t \in (0,T)$ and
  a singular point $x\in M(t)$ such that 
  the tangent flow at $(x,t)$ is a shrinking $\Ss^j\times \RR^{m-j}$ for some
 $j$ with $1\le j\le k$.
 \end{theorem}

 Theorem~\ref{morse} is a special case of Theorem 4.4 in \cite{White_top}.   
 See \cite{HW_Morse} for a simpler, Morse-theoretic proof of Theorem~\ref{morse}.
 
\begin{theorem}[Clearing Out Theorem]\label{clearing-out-theorem} 
Suppose in Theorem~\ref{X-flow} that $X(x)=x/2$.  Then 
$
    \dist(0,K(t)) \to \infty
$
as $t\to\infty$.
\end{theorem}

It is possible that $K(t)$ vanishes in finite time.  Theorem~\ref{clearing-out-theorem}
includes that case: if $K(t)=\emptyset$, then $\dist(0,K(t))=\infty$.

\begin{proof}
If $M_1(\cdot)$ and $M_2(\cdot)$
are weak MCFs in Euclidean space with $M_2(0)$ compact, 
then
\[
   \dist(M_1(t), M_2(t)) := \min_{x\in M_1(t), \, y\in M_2(t)} |x-y|
\]
is a non-decreasing function of $t$.
For MCF of smooth hypersurfaces, this is the standard avoidance principle.
The proof of~\cite[Lemma 4E]{Ilm_mani} gives the general result. 
It follows immediately from the transformation formula~\eqref{tilde-M-flow} 
that if $M_1(\cdot)$ and $M_2(\cdot)$
are renormalized weak MCFs with $M_2(0)$ compact, then
\[
  t \mapsto e^{-t/2}\dist(M_1(t),M_2(t))
\]
is non-decreasing.

Fix a $\tau>0$.  Then $t\mapsto K(\tau+t)$ and $t\mapsto M(t)$ are RMCFs,
so
\begin{equation*}
   \text{$e^{-t/2} \dist(K(\tau+t), M(t))$ is non-decreasing in $t$.}
 \end{equation*}
Since $K(\tau+t)\subset K(t)\subset K(0)$ and since $M(\cdot)= \partial K(\cdot)$, 
we have
\begin{align*}
   \dist(K(\tau+t), M(0)) 
   &\ge \dist(K(\tau+t), M(t))   \\
   &\ge e^{t/2}\dist(K(\tau),M(0)),
\end{align*}
which tends to $\infty$ as $t\to\infty$.
(Note that $\dist(K(\tau), M(0))>0$ since $K(\tau)$ lies in the interior of $K(0)$
 and since $M(0)=\partial K(0)$.)
\end{proof}
\bibliography{HershkovitsWhite}

\begin{bibdiv}
\begin{biblist}

\bib{Ang_tor}{incollection}{
      author={Angenent, Sigurd~B.},
       title={Shrinking doughnuts},
        date={1992},
   booktitle={Nonlinear diffusion equations and their equilibrium states, 3
  ({G}regynog, 1989)},
      series={Progr. Nonlinear Differential Equations Appl.},
      volume={7},
   publisher={Birkh\"auser Boston, Boston, MA},
       pages={21\ndash 38},
      review={\MR{1167827}},
}

\bib{BW_top_prop}{article}{
      author={Bernstein, Jacob},
      author={Wang, Lu},
       title={A topological property of asymptotically conical self-shrinkers
  of small entropy},
        date={2017},
        ISSN={0012-7094},
     journal={Duke Math. J.},
      volume={166},
      number={3},
       pages={403\ndash 435},
         url={https://doi-org.stanford.idm.oclc.org/10.1215/00127094-3715082},
      review={\MR{3606722}},
}

\bib{BW_top_clo}{article}{
      author={Bernstein, Jacob},
      author={Wang, Lu},
       title={Topology of closed hypersurfaces of small entropy},
        date={2018},
        ISSN={1465-3060},
     journal={Geom. Topol.},
      volume={22},
      number={2},
       pages={1109\ndash 1141},
         url={https://doi-org.stanford.idm.oclc.org/10.2140/gt.2018.22.1109},
      review={\MR{3748685}},
}

\bib{CIMW}{article}{
      author={Colding, Tobias~Holck},
      author={Ilmanen, Tom},
      author={Minicozzi, William~P., II},
      author={White, Brian},
       title={The round sphere minimizes entropy among closed self-shrinkers},
        date={2013},
        ISSN={0022-040X},
     journal={J. Differential Geom.},
      volume={95},
      number={1},
       pages={53\ndash 69},
         url={http://projecteuclid.org/euclid.jdg/1375124609},
      review={\MR{3128979}},
}

\bib{CM_generic}{article}{
      author={Colding, Tobias~H.},
      author={Minicozzi, William~P., II},
       title={Generic mean curvature flow {I}: generic singularities},
        date={2012},
        ISSN={0003-486X},
     journal={Ann. of Math. (2)},
      volume={175},
      number={2},
       pages={755\ndash 833},
         url={https://doi.org/10.4007/annals.2012.175.2.7},
      review={\MR{2993752}},
}

\bib{HH}{article}{
      author={Haslhofer, Robert},
      author={Hershkovits, Or},
       title={Singularities of mean convex level set flow in general ambient
  manifolds},
        date={2018},
        ISSN={0001-8708},
     journal={Adv. Math.},
      volume={329},
       pages={1137\ndash 1155},
         url={https://doi-org.stanford.idm.oclc.org/10.1016/j.aim.2018.02.008},
      review={\MR{3783435}},
}

\bib{Hui_Sin}{article}{
      author={Huisken, Gerhard},
      author={Sinestrari, Carlo},
       title={Convexity estimates for mean curvature flow and singularities of
  mean convex surfaces},
        date={1999},
        ISSN={0001-5962},
     journal={Acta Math.},
      volume={183},
      number={1},
       pages={45\ndash 70},
         url={https://doi.org/10.1007/BF02392946},
      review={\MR{1719551}},
}

\bib{Hui_mon}{article}{
      author={Huisken, Gerhard},
       title={Asymptotic behavior for singularities of the mean curvature
  flow},
        date={1990},
        ISSN={0022-040X},
     journal={J. Differential Geom.},
      volume={31},
      number={1},
       pages={285\ndash 299},
         url={http://projecteuclid.org/euclid.jdg/1214444099},
      review={\MR{1030675}},
}

\bib{HW-avoidance}{misc}{
      author={Hershkovits, Or},
      author={White, Brian},
       title={Avoidance for set-theoretic solutions of mean-curvature-type
  flows},
        date={2018},
        note={Preprint at arXiv:1809.03026.},
}

\bib{HW_Morse}{misc}{
      author={Hershkovits, Or},
      author={White, Brian},
       title={A morse theoretic approach to topological changes in mean
  curvature flow},
        date={2018},
        note={In preparation},
}

\bib{HW_X_mean_convex}{misc}{
      author={Hershkovits, Or},
      author={White, Brian},
       title={The regularity and structure theory of advancing fronts motion by
  mean curvature plus a forcing term},
        date={2018},
        note={In preparation},
}

\bib{Ilm_ques}{misc}{
      author={Ilmanen, Tom},
       title={Problems in mean curvature flow},
        date={2003},
        note={Preprint},
}

\bib{Ilm_mani}{incollection}{
      author={Ilmanen, Tom},
       title={The level-set flow on a manifold},
        date={1993},
   booktitle={Differential geometry: partial differential equations on
  manifolds ({L}os {A}ngeles, {CA}, 1990)},
      series={Proc. Sympos. Pure Math.},
      volume={54},
   publisher={Amer. Math. Soc., Providence, RI},
       pages={193\ndash 204},
         url={https://doi.org/10.1090/pspum/054.1/1216585},
      review={\MR{1216585}},
}

\bib{Ilm_sur}{misc}{
      author={Ilmanen, Tom},
       title={Singularities of mean curvature flow of surfaces},
        date={1995},
        note={Preprint},
}

\bib{Ket_shri}{misc}{
      author={Ketover, Dan},
       title={Self-shrinking platonic solids},
        date={2016},
        note={Preprint},
}

\bib{Kap_shri}{article}{
      author={Kapouleas, Nikolaos},
      author={Kleene, Stephen~James},
      author={M\o~ller, Niels~Martin},
       title={Mean curvature self-shrinkers of high genus: non-compact
  examples},
        date={2018},
        ISSN={0075-4102},
     journal={J. Reine Angew. Math.},
      volume={739},
       pages={1\ndash 39},
         url={https://doi-org.stanford.idm.oclc.org/10.1515/crelle-2015-0050},
      review={\MR{3808256}},
}

\bib{Stone}{article}{
      author={Stone, Andrew},
       title={A density function and the structure of singularities of the mean
  curvature flow},
        date={1994},
        ISSN={0944-2669},
     journal={Calc. Var. Partial Differential Equations},
      volume={2},
      number={4},
       pages={443\ndash 480},
         url={https://doi.org/10.1007/BF01192093},
      review={\MR{1383918}},
}

\bib{White_size}{article}{
      author={White, Brian},
       title={The size of the singular set in mean curvature flow of
  mean-convex sets},
        date={2000},
        ISSN={0894-0347},
     journal={J. Amer. Math. Soc.},
      volume={13},
      number={3},
       pages={665\ndash 695},
         url={https://doi.org/10.1090/S0894-0347-00-00338-6},
      review={\MR{1758759}},
}

\bib{White_nature}{article}{
      author={White, Brian},
       title={The nature of singularities in mean curvature flow of mean-convex
  sets},
        date={2003},
        ISSN={0894-0347},
     journal={J. Amer. Math. Soc.},
      volume={16},
      number={1},
       pages={123\ndash 138},
         url={https://doi.org/10.1090/S0894-0347-02-00406-X},
      review={\MR{1937202}},
}

\bib{White_top}{article}{
      author={White, Brian},
       title={Topological change in mean convex mean curvature flow},
        date={2013},
        ISSN={0020-9910},
     journal={Invent. Math.},
      volume={191},
      number={3},
       pages={501\ndash 525},
         url={https://doi.org/10.1007/s00222-012-0397-0},
      review={\MR{3020169}},
}

\bib{White_sub}{article}{
      author={White, Brian},
       title={Subsequent singularities in mean-convex mean curvature flow},
        date={2015},
        ISSN={0944-2669},
     journal={Calc. Var. Partial Differential Equations},
      volume={54},
      number={2},
       pages={1457\ndash 1468},
         url={https://doi.org/10.1007/s00526-015-0831-4},
      review={\MR{3396419}},
}

\bib{White_top_weak}{article}{
      author={White, Brian},
       title={The topology of hypersurfaces moving by mean curvature},
        date={1995},
        ISSN={1019-8385},
     journal={Comm. Anal. Geom.},
      volume={3},
      number={1-2},
       pages={317\ndash 333},
         url={https://doi.org/10.4310/CAG.1995.v3.n2.a5},
      review={\MR{1362655}},
}

\bib{White_stra}{article}{
      author={White, Brian},
       title={Stratification of minimal surfaces, mean curvature flows, and
  harmonic maps},
        date={1997},
        ISSN={0075-4102},
     journal={J. Reine Angew. Math.},
      volume={488},
       pages={1\ndash 35},
         url={https://doi.org/10.1515/crll.1997.488.1},
      review={\MR{1465365}},
}

\end{biblist}
\end{bibdiv}
\bibliographystyle{alpha}
  \end{document}